\documentclass[12pt]{amsart}
\usepackage[a4paper,twoside,left=2cm,right=2cm,top=3.1cm,bottom=2.3cm]{geometry}

\usepackage{amsmath,amssymb,amscd,amsthm}
\usepackage{latexsym}
\usepackage{graphicx}
\usepackage[english]{babel}
\usepackage[latin1]{inputenc}       
\usepackage{textcomp}
\usepackage{times}
\usepackage{colortbl}

\newtheorem{theorem}{Theorem}[section]

\newtheorem{proposition}[theorem]{Proposition}

\newtheorem{remark}[theorem]{Remark}

\def\irr#1{{\rm Irr}(#1)}
\def\irrr#1#2 {\irr {#1 \mid #2}}

\newcommand{\R}{\mathbb R}
\newcommand{\N}{\mathbb N}

\renewcommand{\d}{\mbox{\rm d}}

\begin{document}

\title[]{Log-concavity of solutions of parabolic equations\\ related to the Ornstein-Uhlenbeck operator\\ and applications}

\author[Andrea Colesanti, Lei Qin, Paolo Salani ]{Andrea Colesanti, Lei Qin, Paolo Salani}
\address{Dipartimento di Matematica e Informatica ``U. Dini", Universit\`a degli Studi di Firenze}
\email{andrea.colesanti@unifi.it}
\address{Institute of Mathematics,
Hunan University, Changsha, 410082, China}
\email{qlhnumath@hnu.edu.cn}
\address{Dipartimento di Matematica e Informatica ``U. Dini", Universit\`a degli Studi di Firenze}
\email{paolo.salani@unifi.it}

\begin{abstract}
In this paper, we investigate the log-concavity of the kernel for the parabolic Ornstein-Uhlenbeck operator in a bounded, convex domain. Consequently, we get the preservation of the log-concavity of the initial datum by the related flow.
As an application, we give another proof of a Brunn-Minkowski type inequality for the first eigenvalue of the Ornstein-Uhlenbeck operator and of the log-concavity of the related first eigenfunction (both results have been proved in \cite{CFLS24}, by different methods).
\end{abstract}


\maketitle

\baselineskip18pt

\parskip3pt

\section{Introduction}

Concavity properties of solutions of the elliptic and parabolic type equations and related inequalities have been largely investigated.
In particular, in the parabolic case, the log-concavity of the heat kernel of a convex domain plays a crucial role.
Indeed, exploiting it, in the classical paper \cite{BL76} Brascamp and Lieb proved that the heat flow preserves the log-concavity of the initial datum, and further consequences are a Brunn-Minkowski type inequality for the first non-trivial eigenvalue of the Dirichlet Laplacian and the log-concavity of the corresponding first (positive) eigenfunction of a convex domain (see also \cite{DD1979}).
Very recently, similar results, i.e., a Brunn-Minkowski inequality for the first non-trivial Dirichlet eigenvalue and the the log-concavity of the first eigenfunction, have been established in \cite{CFLS24} for the so called Ornstein-Uhlenbeck operator (that is, roughly speaking, the Laplace operator in the Gauss space), using a method based on the log-concave envelope of the involved functions. A natural question is if the approach used by Brascamp and Lieb in \cite{BL76}, based on the log-concavity of the kernel and on the Pr\'ekopa-Leindler inequality, can be applied in this case as well. In this article we see that the answer is affirmative.

In particular, we consider the following boundary value problem:
\begin{equation}\label{H-PDE2}
\left\{
\begin{aligned}
&\Delta u -(x,\nabla u)=u_t, \quad\mbox{in} \ \Omega \times \{t>0\},  \\
&u=0,  \quad\mbox{on} \ \ \partial\Omega\times \{ t>0 \}, \\
&u(x,0)=u_0(x),\quad\mbox{on} \ \ \Omega \times \{ t=0 \},
\end{aligned}
\right.
\end{equation}
where $\Omega$ is an open bounded subset of $\R^n$, with Lipschitz boundary, and  $u_0\in L^2(\Omega,\gamma)$.
We prove the following.

\begin{theorem}\label{T1}
Let $\Omega$ be an open bounded convex set in $\R^n$, and let $u_0$ be a log-concave function in $\Omega$. Then the solution $u$ of problem \eqref{H-PDE2} is log-concave with respect to $x$, for all $t>0$.
\end{theorem}

The starting point for the proof of this results is that the solution $u=u(x,t)$ can be expressed as follows
\begin{equation}\label{1}
u(x,t)= \int_{\Omega}p_{\gamma}^{\Omega} (x,z;t)u_0(z) d\gamma(z)
\end{equation}
(where $\gamma$ is the Gaussian measure). Here $p_\gamma^\Omega$ is the kernel of the parabolic Ornstein-Uhlenbeck operator, which eventually coincides with the fundamental solution of the equation in \eqref{H-PDE2}, that is, the solution of the following problem:
\begin{equation}\label{parabolic-equation1}
\left\{
\begin{aligned}
& \Delta_x p_\gamma^\Omega(x,z;t)-(x,\nabla_x p_\gamma^\Omega) =\partial_t p_\gamma^\Omega(x,z;t), \quad\forall\ x,z\in\Omega, \ t>0,  \\
& p_\gamma^\Omega(x,z;t)=0, \quad\mbox{if $x\in\partial\Omega$ or $z\in\partial\Omega$, $t>0$},  \\
& p_\gamma^\Omega(x,z;0)d\gamma(z)=\delta(x-z),  \quad  \forall\ x,z\in\Omega,
\end{aligned}
\right.
\end{equation}
where $\delta(x)$ is the Dirac distribution. Very recently, \cite{SUN-Wang2025} proves the partial log-concavity of kernel when $\Omega$ is bounded and convex, that is, fixed $z\in \Omega$ and $t>0$, $p^\Omega_\gamma (x,z;t)$ is log-concave in $x\in \Omega$, and the proof relies on an elliptic method. In the present work, we provide a complete answer through an alternative method.

We point out that, as for the usual heat kernel, although an explicit expression for the entire space $\R^n$ is well known (see for instance \cite{Book-Markov-Diffu,Evans}),
 no explicit formula for $p_\gamma^\Omega$ is available when $\Omega$ is a generic domain.
However, it was shown in \cite{BL76} that the Trotter product formula can be used to deduce the log-concavity of the usual heat kernel when the domain is open, bounded and convex. Applying the same strategy, we will obtain the log-concavity of the kernel of the Ornstein-Uhlenbeck parabolic operator. Precisely, we establish the following theorem, which may be considered in fact the main result of this paper.
\begin{theorem}\label{Lemma-1}
Let $\Omega\subset \R^n$ be a bounded, convex open set. Then, for every $t>0$, the function
$$
(x,z)\, \rightarrow\,p^\Omega_\gamma (x,z;t) e^{-|z|^2/2},\quad(x,z)\in\Omega\times\Omega
$$
is log-concave.
\end{theorem}
Then, this fact, along with the representation formula \eqref{1}, ensures Theorem \ref{T1}, by a standard argument based on Pr\'ekopa's theorem.

As a consequence of Theorem \ref{T1}, we will obtain new proofs, {\it \`a la} Brascamp and Lieb, of two results established in \cite{CFLS24}: the log-concavity of the first eigenfunction of the Ornstein-Uhlenbeck operator in convex domains and a Brunn-Minkowski type inequality for the relevant eigenvalue. 
Let us describe in more details these results.

Given an open and bounded set $\Omega$ in $\R^n$, with Lipschitz boundary, there exists a unique positive number $\lambda_\gamma(\Omega)$, such that the following problem
\begin{equation}\label{Lambda-Gauss-intro}
\left\{
\begin{aligned}
&\Delta u-(x,\nabla u)=-\lambda_\gamma(\Omega) u\quad\mbox{in $\Omega$,}\\\
& u>0\quad\mbox{in $\Omega$,}\\
& u=0\quad\mbox{on $\partial \Omega$},
\end{aligned}
\right.
\end{equation}
admits a solution. $\lambda_\gamma(\Omega)$ is called the first eigenvalue of the Ornstein-Uhlenbeck operator and can be also characterized as follows,
$$
\lambda_\gamma(\Omega)=\inf\left\{\frac{\int_\Omega |\nabla v|^2d\gamma}{\int_\Omega v^2d\gamma}\,: v\in W_0^{1,2}(\Omega),\, v\not\equiv 0\right\}\,,
$$
where $d\gamma$ stands for integration with respect to the Gauss measure $\gamma$, defined as
$$
\gamma(A)=\frac1{(2\pi)^{n/2}}\int_A e^{-|x|^2/2}\d x.
$$
The solution of problem \eqref{Lambda-Gauss-intro} is unique up to the multiplication by a (positive) number, and it is called eigenfunction of the eigenvalue $\lambda_\gamma(\Omega)$.
Similarly to the case of the Laplace operator, we can reprove the two following properties.

\begin{theorem}[Theorem 1.2 in \cite{CFLS24}]\label{Application2} Let $s\in [0,1]$. Let $\Omega_0$, $\Omega_1$, and
\[
\Omega_s=(1-s)\Omega_0+s\Omega_1:=\{(1-s)x_0+sx_1:\ x_0\in \Omega_0,\ x_1\in \Omega_1\}
\]
be open bounded subsets of $\R^n$, $(n\geq 2)$, with Lipschitz boundary. Then
\[
\lambda_\gamma(\Omega_s)\leq (1-s)\lambda_\gamma(\Omega_0)+s\lambda_\gamma(\Omega_1).
\]
\end{theorem}

\begin{theorem}[Theorem 1.7 in \cite{CFLS24}]\label{Application1}
Let $\Omega$ be an open bounded and convex subset of $\R^n$. Let $u$ be a solution of problem \eqref{Lambda-Gauss-intro}. Then $u$ is log-concave in $\Omega$.
\end{theorem}

We mention that in \cite{CFLS24} the authors provide two proofs of the previous result, one based on the {\it log-concave envelope}, and another one based on the {\it constant rank theorem} (see also \cite{CCLaMP,Colesanti2025,LeiQin2025} for related topics).



\medskip

The organization of the paper is as follows. In Section \ref{Se2}, we list some basic facts and introduce some notations. In Section \ref{Se3}, we introduce the Ornstein-Uhlenbeck semigroup and the Trotter product formula. In Section \ref{Se4}, we give the proof of  Theorem \ref{Lemma-1} and Theorem \ref{T1}. In Section \ref{Se5}, we prove Theorem \ref{Application1} and Theorem \ref{Application2}. In the Appendix, we present the semigroup theory for Ornstein-Uhlebeck in the class of bounded Lipschitz domains.

\medskip

{\bf Acknowledgments.} The second author would like to thank the Department of Mathematics and Computer Science of the University of Florence for the hospitality.
The first author was partially supported by INdAM through different GNAMPA projects, and by the Italian "Ministero dell'Universit\`a e Ricerca" and EU through different PRIN projects. The Third author has been partially financed by European Union -- Next Generation EU -- through the project "Geometric-Analytic Methods for PDEs and Applications (GAMPA)", within the PRIN 2022 program (D.D. 104 - 02/02/2022 Ministero dell'Universit\`a e della Ricerca).

\section{Preliminaries}\label{Se2}

We work in the $n$-dimensional Euclidean space $\mathbb{R}^n$, $n\geq 2$, and we denote by $|\cdot|$ the Euclidean norm, and by $(\cdot,\cdot)$ the standard scalar product.  Given $\Omega\subset\R^n$, let $C_{c}^{\infty}(\Omega)$ be the set of functions from $C^{\infty}(\Omega)$ having compact support in $\Omega$.

We denote the Gauss probability space by $( \mathbb{R}^{n}, \gamma)$; the measure $\gamma$ is given by
$$
\gamma(E)= \frac{1}{(2\pi)^{n/2}} \int_{E} e^{-|x|^2/2} dx,
$$
for any measurable set $E\subseteq \mathbb{R}^{n}$. We use the notation $dx$ for integration with respect to Lebesgue measure in $\R^n$, while $d\gamma$ stands for integration with respect to $\gamma$, i.e., $d\gamma(x)=(2\pi)^{-n/2} e^{-|x|^2/2}dx$.

Let $\Omega\subseteq \R^n$ be an open set. We use the standard notations $L^2(\Omega)$, $W^{1,2}(\R^n)$, $W^{2,2}(\Omega)$ and $W^{1,2}_0(\Omega)$ for Lebesgue and Sobolev spaces. We also set
\begin{equation}\nonumber
L^2(\Omega,\gamma)=\left\{u:\Omega\rightarrow \R\colon \mbox{$u$ is measurable and}\ \|u\|^{2}_{2,\gamma}:=\int_{\Omega} |u(x)|^2 d\gamma(x) <+\infty \right\}.
\end{equation}
Similarly, we define the space
\begin{equation}\nonumber
W^{1,2}(\Omega,\gamma)=\left\{u\in W_{\text{loc}}^{1,2}(\R^n)\colon\|u\|^2_{W^{1,2}(\Omega,\gamma)}:=\int_{\Omega}\left(|\nabla u(x)|^{2} +|u(x)|^2\right) d\gamma (x)<+\infty \right\}.
\end{equation}
$W^{1,2}_0(\Omega,\gamma)$ denotes the closure of $C^\infty_c(\Omega)$ in $W^{1,2}(\Omega,\gamma)$.

The scalar product of the Hilbert space $L^2(\Omega,\gamma)$ is defined as follows,
\begin{equation}
(u,\upsilon)_{L^2}=\int_{\Omega}u\upsilon d\gamma,\quad \forall\, u,\upsilon\in L^2(\Omega,\gamma).
\end{equation}
By the weighted Poincar\'{e} inequality, the associated scalar product of the Hilbert space $W^{1,2}_0(\Omega,d\gamma)$ can be given as
\[
(u,\upsilon)_{W_0^{1,2}}=\int_{\Omega}\nabla u\cdot \nabla \upsilon d\gamma,\quad \mathrm{for\ any}\ u,\upsilon\in W^{1,2}_0(\Omega,\gamma).
\]

The spaces $W^{1,2}(\Omega,\gamma)$ and $W^{2,2}(\Omega,\gamma)$ are defined in a similar way. These Sobolev spaces are all separable Hilbert spaces with respect to the Gaussian measure (see for instance \cite{Tero1994,Radulescu2019} and references therein). When $\Omega$ is bounded the density function $e^{-|x|^2/2}$ has positive upper and lower bounds in $\Omega$, hence it is obvious that $L^2(\Omega,\gamma)=L^2(\Omega)$, $W^{1,2}(\Omega,\gamma)=W^{1,2}(\Omega)$, and so on. However, we only have $W^{1,2}(\Omega)\subset W^{1,2}(\Omega,\gamma)$ and $W^{1,2}_0(\Omega)\subset W^{1,2}_0(\Omega,\gamma)$ when $\Omega$ is unbounded.

\subsection{The Ornstein-Unlenbeck operator and its spectrum}\label{Se2-1}

Let $\Omega$ be an open subset of $\R^n$. The Ornstein-Uhlenbeck operator $L_\Omega$ relative to $\Omega$, is defined as follows
\[
L_\Omega u=\Delta u-(x,\nabla u).
\]
It will be important to identify the domain of this operator. When $\Omega=\R^n$, we set
$$
\text{Dom}(L_{\R^n}):=W^{2,2}(\R^n,\gamma).
$$
For a general open set $\Omega$, we just note now that $L_\Omega$ is well defined on $W^{2,2}(\Omega,\gamma)$; later on we restrict $L_\Omega$ to a slightly different space of functions.

Integration by parts shows that $L_\Omega$ is self-adjoint with respect to the Gaussian measure, in the following sense: for every $\phi,\psi \in C_{c}^\infty (\Omega)$,
\[
\int_{\Omega} \phi L_\Omega\psi d\gamma=-\int_{\Omega} (\nabla \phi, \nabla \psi)d\gamma=\int_{\Omega} \psi L_\Omega\phi d\gamma.
\]

\medskip

For an open bounded set $\Omega\subset\R^n$, with Lipschitz boundary, it is well-known (and it follows via, say, a variational argument), that there exists a unique positive number $\lambda_{\gamma}(\Omega)$ such that the boundary value problem \eqref{Lambda-Gauss-intro}
admits a solution $u\in W^{1,2}_0(\Omega,\gamma)$. Moreover, the solution is unique up to multiplication by positive constants (see \cite[Theorem 2.1]{CFLS24}). When no ambiguity may arise, we will write $\lambda_\gamma$ instead of $ \lambda_\gamma(\Omega)$.

In view of the previous definition, $\lambda_\gamma$ is an eigenvalue of the operator $L_\Omega$. By standard spectral theory (see \cite[Theorem 4.46]{Arendt-Urban2023}), the Ornstein-Uhlenbeck operator has, more generally, a countable set of eigenvalues $\Sigma=\{\lambda_i\in \R^{+}| 0<\lambda_\gamma=\lambda_1<\lambda_2\leq \cdots \leq \lambda_k\leq\cdots \}$, with $\lambda_k\rightarrow +\infty\ \text{as}\ k\rightarrow +\infty$, whose normalized eigenfunctions $\{ \phi_k\}$ are mutually orthogonal and span $W^{1,2}_0(\Omega,\gamma)$ (in $L^2(\Omega,\gamma)$ as well). By the regularity theory of elliptic partial differential equations, $\phi_k \in C^{\infty}(\Omega)\cap C^0(\overline{\Omega})$ for every $k$ (cf. \cite{Evans}).

\section{The Ornstein-Uhlenbeck semigroup and the Trotter product formula}\label{Se3}

In this section we show that, if $\Omega$ is a bounded open set with Lipschitz boundary, the semigroup generated by $L_\Omega$ can be obtained from the semigroup relative to $L_{\R^n}$ via Trotter product formula. Our general references for the theory of semigroups are \cite{Arendt-handbook,Pazy1983}.

Let $(P_\gamma(t))_{t\geq 0}$ be the Ornstein-Uhlenbeck semigroup with generator operator $-L_{\R^n}$. It is well-known (see for instance \cite{Book-Markov-Diffu}) that $P_\gamma(t)$ admits a kernel density, called Mehler kernel, with respect to the Lebesgue measure in $\R^n$. This means that there exists a function $p=p(x,z;t)$, with $x,z\in\R^n$, and $t>0$, such that
\begin{equation}\label{Pf}
[P_\gamma(t)f](x)=\int_{\R^n}p(x,z,t)f(z)dz\,,
\end{equation}
for any $f\in L^2(\R^n)$.

\subsection{Properties of the Mehler kernel}

The Mehler kernel can be written in an explicit form: for every $t>0$ and $(x,z)\in \R^n \times \R^n$, we have
\begin{equation}\label{heat-kernel-2}
p(x,z;t)=\frac{1}{(2\pi(1-e^{-2t}))^{n/2}} \mbox{exp} \left( -\frac{|z-e^{-t}x|^2}{2(1-e^{-2t})}\right).
\end{equation}
We observe that $p$ is log-concave with respect to $(x,z)$, for every $t>0$. In other words, for every $t>0$ the function:
$$
(x,z)\, \rightarrow\,\log(p(x,z;t)),\quad(x,z)\in\R^n\times\R^n,
$$
is concave.
\begin{remark}
It can be noticed that $p$ shares indeed some stronger concavity property than log-concavity, so that, as it was done in \cite{IST2020, IST2024} for the usual heat flow, we may conjecture that such stronger properties are preserved by the Gaussian heat flow, and that this would possibly lead to an improvement of Theorem \ref{Application1}.
\end{remark}

Next, notice that
\begin{equation}\label{kernels relation}
\mbox{exp} \left( -\frac{|z-e^{-t}x|^2}{2(1-e^{-2t})}\right)=
\text{exp} \left( -\frac{|x|^2-2e^t (x,z)+|z|^2}{2(e^{2t}-1)}\right)\,e^{-|z|^2/2}.
\end{equation}
Hence, \eqref{Pf} can be rewritten as follows with respect to the Gaussian measure:
\begin{equation}\label{aggiunta}
[P_\gamma(t)f](x)=\int_{\R^n}p_{\gamma}(x,z,t)f(z)d\gamma(z),
\end{equation}
for every $f\in L^2(\R^n,\gamma)$, where
\begin{equation}\label{heat-kernel-1}
p_\gamma(x,z;t):= \frac{1}{(1-e^{-2t})^{n/2}} \text{exp} \left( -\frac{|x|^2-2e^t (x,z)+|z|^2}{2(e^{2t}-1)}\right).
\end{equation}
By a change of variable, \eqref{aggiunta} and \eqref{heat-kernel-1} lead to:
\[
[P_\gamma(t)f](x)=\int_{\R^n}f(e^{-t}x+\sqrt{1-e^{-2t}}y) d\gamma(y), \ \ \ \ t\geq 0,\ \ x\in \R^n.
\]
It can be verified that
\begin{equation}
\int_{\R^n}\int_{\R^n} p_\gamma^2(x,z;t)d\gamma(x)d\gamma(z)<\infty.
\end{equation}
This implies that $P_\gamma(t)(t>0)$ has the Hilbert-Schmidt property (cf. \cite[p.104 and Section A.6, p.483]{Book-Markov-Diffu}).

The following simple properties will be repeatedly used in this article, so we enlighten them in the following proposition.
\begin{proposition}\label{log-concavity proposition} For every $t>0$, the function
$$
(x,z)\, \rightarrow\,p_\gamma (x,z;t) e^{-|z|^2/2},\quad(x,z)\in\R^n\times\R^n
$$
is log-concave. Moreover, if $g\colon\R^n\to\R$ is log-concave, then, for every $t>0$, the function $G\colon\R^n\to\R$ defined by
$$
G(x)=\int_{\R^n} g(z)p_\gamma(x,z;t)d\gamma(z),
$$
is log-concave.
\end{proposition}

\begin{proof}
The first claim is trivially verified. The second claim is a consequence of the first one and of Pr\'ekopa's Theorem (see for instance \cite[Corollary 3.5 ]{BL76}).
\end{proof}

\subsection{Trotter product formula}\label{Se3-1}
Let $\Omega\subset \R^n$ be an open subset of $\R^n$. We introduce the set
$$
\text{Dom}(L_\Omega):=\{ u\in W^{1,2}_0(\Omega,\gamma):\ \text{the\ weak\ Laplacian}\ \Delta u\ \text{of}\ u\ \text{belongs\ to}\ L^2(\Omega,\gamma)\}.
$$
Integration by parts shows that $L_\Omega$ is a self-adjoint operator in $\text{Dom}(L_\Omega)$, with respect to the Gaussian measure: for every $u,\upsilon \in \text{Dom}(L_\Omega)$, we have
\[
(L_\Omega u,\upsilon)=(u,L_\Omega \upsilon)=-\int_{\Omega}(\nabla u, \nabla \upsilon) d\gamma.
\]

By Hille-Yosida theorem, we deduce that the operator $-L_\Omega$ generates a symmetric and ultracontractive $\mathcal{C}_0$-semigroup, denoted by $e^{-tL_\Omega}$ (cf. \cite{Pazy1983}). Using the same arguments as in \cite[Chapter 2]{DaviesBook}, we conclude that the semigroup $e^{-tL_\Omega}$ can be written in the form
\[
e^{-tL_\Omega}f=\int_\Omega p^\Omega_\gamma (x,z;t) f(z) d\gamma(z)
\]
for any $f\in L^2(\Omega,\gamma)$, where the kernel $p^{\Omega}_\gamma(x,z;t)\in L^2(\Omega \times \Omega,\gamma)$, for every $t>0$. In fact, in Section \ref{Se4}, we will prove that the kernel is also the solution of problem \eqref{parabolic-equation1}. The aim of this part is to use the Trotter product formula to write $e^{-tL_\Omega}$ as a limit of a suitable sequence of semigroups.

Let $\Omega\subset \R^n$ be an open subset of $\R^n$, we set
$$
H_1:=\{u\in W^{1,2}(\R^n,\gamma):\ u(x)=0,\quad\forall\ x\in \Omega^c\}.
$$
The following remark is well-known.

\begin{remark}
If $\Omega$ has Lipschitz boundary, every function $u\in W_0^{1,2}(\Omega,\gamma)$ can be extended to $H_1$ as follows
\[
\tilde{u}=\left\{
\begin{aligned}
& u,\quad x\in \Omega, \\
&  0,\quad  x\notin \Omega.
\end{aligned}
\right.
\]
Vice versa, if $u\in H_1$ then its restriction to $\Omega$ belongs to $W^{1,2}_0(\Omega,\gamma)$.
\end{remark}
From now on, we assume that $\partial\Omega$ has Lipschitz boundary.

Let $A=L_{\R^n}$. To $A$ we associate a bilinear form $A(u,\upsilon)$ in $H_1$ , defined by
$$
A(u,\upsilon):=\int_{\R^n} (\nabla u, \nabla \upsilon) d\gamma(x).
$$

The identity operator $\mathbf{I}_0\colon H_1\to W^{1,2}_0(\Omega,\gamma)$ can be written as $\mathbf{I}_0f=\chi_\Omega f$ for all $f\in H_1$, where $\chi_\Omega$ is the characteristic function of $\Omega$. $\mathbf{I}_0$ is surjective, by the previous remark. As stated in \cite{Book-Engel2000,Pazy1983}, $\mathbf{I}_0$ can be viewed as a strongly continuous semigroup with generator $B=0$, that is $e^{tB}:=\mathbf{I}_0$ and $\text{Dom}(B)=H_1$.

The sum of operators $C:=A+B$ is defined via the sum of the bilinear forms of $A$ and $B$, that is,
\[
C(u,\upsilon)=(A+B)(u,\upsilon)=A(u,\upsilon)+B(u,\upsilon),
\]
for all $u,\upsilon\in H_1$. Moreover, $\text{Dom}(C):=\text{Dom}(A)\cap \text{Dom}(B)=H_1$. Clearly, $C$ is also a self-adjoint operator in $H_1$. Therefore, by (trivially) extending the functions of the space $W^{1,2}_0(\Omega,\gamma)$ to be zero outside of $\Omega$, the definitions of the operators $C$ and $L_\Omega$ are consistent. Moreover, $e^{-t(A+B)}$ verifies
\begin{equation}\label{Markov-kernel}
[e^{-tC}]f(x)=\left\{
\begin{aligned}
& [e^{-tL_\Omega}f](x),\quad x\in \Omega, \\
&  0,\quad \ \ \ \ \ \ \ \ \ \ \ \ \ \ \ \ \ x\notin \Omega,
\end{aligned}
\right.
\end{equation}
for all $f\in H_1$, also in $ W^{1,2}_0(\Omega,\gamma)$.

By Trotter product formula \cite[formula (3), page vii]{Book-Totter-product-formula}, we have
\begin{equation}
e^{-t(A+B)}f =\mathop{\lim}_{m\rightarrow \infty} (e^{-\frac{t}{m}A}e^{-\frac{t}{m}B})^m f=\mathop{\lim}_{m\rightarrow \infty} (P_\gamma (t/m)\,\mathbf{I}_0)^m f,
\end{equation}
for all $f\in \text{Dom}(A)\cap \text{Dom}(B)$. From now on, for more clarity and simplicity, we will write $P_\gamma (t/m)|_\Omega$ in place of $P_\gamma (t/m)I_0$. Therefore, for $f\in W^{1,2}_0(\Omega,\gamma)$,
\begin{equation}\label{Totter}
e^{-tL_\Omega}f=\mathop{\lim}_{m\rightarrow \infty} (P_\gamma (t/m)|_\Omega)^m f,
\end{equation}
where the convergence is in the sense of $L^2(\Omega,\gamma)$-norm. Hence we can express $e^{-tL_\Omega}$ as a limit of integral operators.

\begin{remark}\label{remark}
In formula \eqref{Totter}, the domain $W_0^{1,2}(\Omega,\gamma)$ can be extended to the Sobolev space $L^2(\Omega,\gamma)$, since the eigenfunctions $\{ \phi_k\}_{k\geq 1}$ span $W^{1,2}_0(\Omega,\gamma)$, also in $L^2(\Omega,\gamma)$.
\end{remark}

For the theory of semigroups for the Ornstein-Uhlenbeck operator, we recommend \cite{Book-Markov-Diffu,Cappa2019} and references therein. However, since the Ornstein-Uhlenbeck operator is linear, all standard references on the theory of semigroups for linear operators in the context of weighted Sobolev spaces, such as \cite{DaviesBook,Book-Engel2000,Pazy1983} and so on, can be directly consulted.

\subsection{Relationship between Ornstein-Uhlenbeck operator and Schr\"{o}dinger operators}\label{Se3-2}

It is well-known how to connect the Ornstein-Uhlenbeck operators and the harmonic oscillator operators, that is, operators of the form $L_V:=-\Delta+V(x)$.

Given $u\in C^2(\Omega)$, we set
\[
w=u e^{-|x|^2/4}.
\]
By a direct calculation we see that $u$ is a solution of problem \eqref{Lambda-Gauss-intro} if and only if $w$ is a solution of
\begin{equation}\nonumber
\left\{
\begin{aligned}
& -\Delta w + V(x)w=\lambda_\gamma w, \ \ \   \\
& w>0\quad\mbox{in $\Omega$}, \\
& w=0\quad\mbox{on $\partial\Omega$},
\end{aligned}
\right.
\end{equation}
where
$$
V(x)=\left(-\frac{n}{2}+\frac{|x|^2}{4} \right)\,.
$$
Moreover, since $-\Delta$ is symmetric with respect to $dx$, then $-\Delta$ is symmetric with respect to $d\mu:=e^{-V}dx$, and the map $f\rightarrow e^{V/2}$ is an isometry between $L^2(\Omega,dx)$ and $L^2(\Omega,d\mu)$. Under these conditions, the operators $-\Delta$ and $-\Delta+V$ have the same spectral properties. In particular, when $\Omega$ is an open and bounded set, the spectrum of $-\Delta$ is discrete and the same holds for the spectrum of $-\Delta+V$. We refer to \cite{Colesanti2025} for more details.


\section{Log-concavity of solution of the Fokker-Planck equation}\label{Se4}

Let us consider the following problem:
\begin{equation}\label{H-PDE1}
\left\{
\begin{aligned}
& \Delta u -(x,\nabla u)=\partial_t u, \quad & \mbox{in $\R^n\times \{t>0\}$,}\\
& u(x,0)=u_0(x),  \quad & \mbox{on $\R^n \times \{ t=0 \}$.}
\end{aligned}
\right.
\end{equation}

It is well-known that the solution to problem \eqref{H-PDE1} is log-concave with respect to the spatial variable $x$, if the initial datum $u_0$ is log-concave (see for instance \cite{Book-Markov-Diffu}). For the sake of completeness, we provide a concise proof of this result.

\begin{proposition}
Let $u_0\in L^2(\R^n,\gamma)$. Then the function $u\colon\R^n\times[0,\infty)\to\R$, defined by
\[
u(x,t)=\int_{\R^n} u_0(e^{-t}x+\sqrt{1-e^{-2t}}y) d\gamma(y),
\]
is a solution of problem \eqref{H-PDE1}. Furthermore, if $u_0$ is log-concave in $\R^n$, then $u(\cdot,t)$ is log-concave for every $t>0$.
\end{proposition}

\begin{proof}[\bf Proof]
By \eqref{heat-kernel-1}, we can directly check that the kernel $p_\gamma$ solves the parabolic equation in problem \eqref{parabolic-equation1}. As a consequence, the function $u$ defined by
\begin{equation}\label{Pt}
\begin{aligned}
u(x,t)=\int_{\R^n} u_0(z) p_\gamma(x,z;t) d\gamma(z)
\end{aligned}
\end{equation}
is a solution of the PDE in problem \eqref{H-PDE1}. Moreover, by the change of variable $z=e^{-t}x+\sqrt{1-e^{-2t}}y$, we obtain the following expression for $u$:
\begin{equation}\nonumber
\begin{aligned}
u(x,t)=\int_{\R^n} u_0(e^{-t}x+\sqrt{1-e^{-2t}}y) d\gamma(y).
\end{aligned}
\end{equation}
Clearly, letting $t=0$, we have $u(x,t)=u_0$. The previous facts show that $u$ is a solution of the equation in \eqref{H-PDE1}.

Finally, if $u_0$ is log-concave, then the log-concavity of $u(\cdot,t)$, for every $t>0$, follows from \eqref{Pt} and Proposition \ref{log-concavity proposition}.
\end{proof}

Now we are ready to prove Theorem \ref{Lemma-1}, a crucial result for this paper.

\begin{proof}[\bf Proof of Theorem \ref{Lemma-1}] Let $m\in\mathbb{N}$. For a fixed $t>0$, we define, for $x,z\in\Omega$,
\begin{align*}
f_m(x_0,x_1,\cdots,x_m;t)=\prod_{i=1}^{m}p_\gamma (x_{i-1},x_i,t/m),
\end{align*}
where $x_0=x,x_m=z$, $x_i\in \R^n$, $1\leq i\leq m-1$. It is obvious that $f_m(x,z;t)\in C^{\infty}(\R^n \times \R^n \times (0,+\infty))$ for every $m\in \N$.

We also set:
\begin{eqnarray*}
g_m(x,z;t)=\int_{\Omega}\cdots\int_{\Omega} f_m(x,x_1,\cdots,x_{m-1},z;t) d\gamma(x_1)\cdots d\gamma(x_{m-1})
\end{eqnarray*}
(we will observe within few lines that this quantity is finite). The regularity of $f_m$ implies that $g_m\in C^{\infty}(\Omega \times \Omega \times (0,+\infty))$ for every $m\in \N$.

By the semigroup property of the kernel $p_\gamma$, we know that
\begin{equation}\label{Chapman-Kolmogorov-equations}
\int_{\R^n} p_\gamma (x_{k-1},x_k;t/m) p_\gamma(x_k,x_{k+1};t/m)d\gamma(x_k)= p_\gamma(x_{k-1},x_{k+1};2t/m),\quad k=1,\ldots,m-1.
\end{equation}
Therefore $g_m (x,z;t) \leq p_\gamma (x,z;t)<+\infty$ and
\begin{equation}\label{aggiunta10}
\int_\Omega g_m (x,z;t)d\gamma(z) \leq \int_{\R^n} p_\gamma (x,z;t)d\gamma(z)=1,
\end{equation}
for all $t>0$ and $m\in \N$. Moreover, by formula \eqref{Chapman-Kolmogorov-equations}, we can check that $g_{2^m}(x,z;t)$ is a point-wise decreasing sequence of functions, with respect to the index $m$. As a consequence, there exists a function $g=g(x,z;t)\geq 0$, such that
\[
g(x,z;t)=\mathop{\lim}_{m \rightarrow \infty} g_{2^m}(x,z;t),\quad \mathrm{for\ all }\ x,z\in \Omega, \ t>0.
\]
By the Lebesgue dominated convergence theorem, for all $\varphi \in L^2(\Omega,\gamma)$, we have
\[
\int_\Omega g_{2^m} (x,z;t)\varphi (z)d\gamma(z) \xrightarrow{L^2(\Omega,\gamma)}\int_{\Omega} g(x,z;t)\varphi(z)d\gamma(z),\quad \mathrm{as}\ m\rightarrow +\infty.
\]
On the other hand, by the definition of $g_m$, we know that for every $m$ and for every $x\in \Omega$,
\begin{equation}\label{aggiunta100}
\int_\Omega g_{m} (x,z;t)\varphi (z)d\gamma(z)=(P_\gamma (t/m))^m\varphi(x).
\end{equation}
By formula \eqref{aggiunta100}, \eqref{Totter} and Remark \ref{remark}, we know that for all $\varphi \in L^2(\Omega,\gamma)$
\[
\int_\Omega g_{m} (x,z;t)\varphi (z)d\gamma(z) \xrightarrow{L^2(\Omega,\gamma)}\int_{\Omega} p^\Omega_\gamma(x,z;t)\varphi(z)d\gamma(z),\quad \mathrm{as}\ m\rightarrow +\infty.
\]
Thus, we deduce that for all $\varphi \in L^2(\Omega,\gamma)$, $x\in \Omega$,
\[
\int_{\Omega} g(x,z;t)\varphi(z)d\gamma(z)=\int_{\Omega} p^\Omega_\gamma(x,z;t)\varphi(z)d\gamma(z),
\]
which implies that
\[
g(x,z;t)=p^\Omega_\gamma(x,z;t),\quad \mbox{for\ a.e.}\ x,z\in \Omega,\ t>0.
\]
This yields
\[
g_{2^m} (x,z;t)\rightarrow p^\Omega_\gamma (x,z;t),\quad \mathrm{as}\ m\rightarrow +\infty,
\]
for all $x,z\in \Omega$ and $t>0$.

For $t>0$, the function
\begin{eqnarray*}
H_0(x_0,\dots,x_m;t)=
f(x_0,\dots,x_m;t) e^{-\frac{|x_1|^2+|x_2|^2+\dots+|x_m|^2}{2}}
\end{eqnarray*}
is log-concave in $(x_0,\dots,x_m)$, by Proposition \ref{log-concavity proposition}. Hence the function
\begin{equation*}
H_1(x_0,\dots,x_m;t)=H_0(x_0,\dots,x_m;t)\chi_\Omega(x_1)\dots\chi_\Omega(x_{m-1})
\end{equation*}
is log-concave, as product of log-concave functions, because $\Omega$ is convex. By Pr\'{e}kopa's theorem, integration with respect to $x_1,x_2,\cdots,x_{m-1}$ will leave this property unchanged for the result of integration, i.e. for the function
$$
g_m(x,z;t)e^{-|z|^2/2}.
$$
Therefore, $p^\Omega_{\gamma}(x,z;t)e^{-|z|^2/2}$ is log-concave with respect to $(x,z)\in \Omega$, for fixed $t>0$, as a limit of log-concave functions.
\end{proof}

\begin{remark} From the probabilistic point of view, the kernel $p_\gamma^\Omega$ is the transition probability density of Brownian motions, killed when exiting the set $\Omega$. There exist probabilistic methods to obtain point-wise estimates, especially for the heat kernel associated with the Dirichlet Laplacian. This is an idea due to Kac \cite{Kac1966,Kac1950} and has been used extensively.
\end{remark}

\begin{proposition}\label{p3} Let $\Omega \subset \R^n$ be an open bounded set, with Lipschitz boundary. Then, for every $f\in L^2(\Omega,\gamma)$ and for every $t>0$, the following inequality holds
\[
|e^{-tL_\Omega}f| \leq (P_\gamma (t)|_\Omega)|f|.
\]
\end{proposition}

\begin{proof}[\bf Proof.] It is enough to prove the statement for $f\geq 0$. In the notations of the proof of Theorem \ref{Lemma-1}, using \eqref{Totter}, \eqref{aggiunta10} and the fact that $g_{2^m}$ converges decreasing to $p_\gamma^\Omega$, we obtain
\[
\int_\Omega p_\gamma^{\Omega}(x,z;t) f(z)d\gamma(z)\leq \int_\Omega p_\gamma(x,z;t)f(z)d\gamma(z),
\]
which completes the proof.
\end{proof}

A brief review of the definition of the sequences $\{ \lambda_k \}$ and $\{ \phi_k \}$ is necessary here. As previously stated in Section \ref{Se2-1}, operator $L_\Omega$ has a countable set of eigenvalues $\{\lambda_k\}_{k\in\N}$, whose normalized eigenfunctions $\{ \phi_k\}_{k\in\N}$ span $W^{1,2}_0(\Omega,\gamma)$ (in $L^2(\Omega,\gamma)$ as well) and are mutually orthogonal. By "normalized" we mean
$$
\|\phi_k\|_{L^2(\Omega,\gamma)}=1\,.
$$

\begin{proposition}\label{p2}
Let $\Omega$ be an open, bounded set in $\R^n$, with Lipschitz boundary. Then
\[
\|\phi_k\|_{L^{\infty}(\Omega)}\leq c\lambda_k^{n/4},
\]
where $c:=c(\Omega,n)$ is a positive constant. Moreover, for all $t>0$, we have
\[
\mathop{\sum}^{\infty}_{k=1} e^{-t\lambda_k }< +\infty.
\]
\end{proposition}

\begin{proof}[\bf Proof]
 By the spectral mapping theorem \cite[Theorem 2.4]{Pazy1983}, we know that the eigenvalues of the self-adjoint operator $-L_\Omega$ are the eigenvalues of the semigroup $e^{-tL_\Omega}$ (which is also a self-adjoint operator) in the Hilbert space $L^2(\Omega,\gamma)$, and the corresponding eigenfunctions coincide. So, we have
\begin{equation}\label{Section5-1}
e^{-t\lambda_k}\phi_k=e^{-tL_\Omega} \phi_k,\quad\forall\, k\in\N.
\end{equation}
Therefore, by Proposition \ref{p3} and H\"{o}lder inequality, we have
\begin{align*}
|e^{-t\lambda_k}\phi_k|&=|e^{-tL_\Omega} \phi_k|\leq \left|\int_\Omega p_\gamma(x,z;t)\phi_k(z)d\gamma(z)\right|\\
&\leq \left( \int_\Omega [p_\gamma(x,z;t)]^2 d\gamma(z)\right)^{1/2}\left( \int_\Omega \phi_k(z)^2 d\gamma(z)\right)^{1/2}\\
&\leq \left( \int_{\Omega} [p_\gamma(x,z;t)]^2 d\gamma(z)\right)^{1/2},
\end{align*}
where the last inequality follows from $\| \phi_k\|_{L^2(\Omega,\gamma)}=1$.

By the definition of $p$ and $p_\gamma$, we have
\begin{equation}\nonumber
\int_\Omega p_\gamma (x,z;t)d\gamma(z)=\int_\Omega p(x,z;t)dz\leq \int_{\R^n} p(x,z;t)dz =1.
\end{equation}
Since $\Omega$ is bounded, by formula \eqref{heat-kernel-2}, there exists a positive constant $c_0:=c(\Omega,n)$ such that for every $x,z\in \overline{\Omega}$ and $t>0$, we have
\[
p_\gamma(x,z;t)\leq  \frac{c_0}{(1-e^{-2t})^{n/2}} .
\]
From the previous considerations, we deduce that
\begin{equation}\label{21}
|\phi_k|\leq e^{t\lambda_k}\left( \int_\Omega [p_\gamma(x,z;t)]^2 d\gamma(z)\right)^{1/2}\leq c_0\frac{e^{t\lambda_k}}{(1-e^{-2t})^{n/4}}.
\end{equation}
The previous inequality holds for every $t>0$. Define
\[
h(t):=\frac{e^{t\lambda_k}}{(1-e^{-2t})^{n/4}},\quad t>0.
\]
By a direct calculation, its minimum value is attained when $t_0=\ln \sqrt{1+\frac{n}{2\lambda_k}}$. Then, we have
\[
h(t_0)=\frac{e^{t_0\lambda_k}}{(1-e^{-2t_0})^{n/4}}=\left(1+\frac{n}{2\lambda_k} \right)^{\frac{\lambda_k}{2}} \left(1+\frac{2\lambda_k}{n} \right)^{\frac{n}{4}}.
\]
Since $\lambda_k\rightarrow +\infty$ as $k\rightarrow \infty$, we have
\[
\mathop{\lim}_{k\rightarrow \infty}\left(1+\frac{n}{2\lambda_k} \right)^{\frac{\lambda_k}{2}}= e^{\frac{n}{4}}.
\]
Therefore, by formula \eqref{21}, we deduce that
\begin{equation}\nonumber
\| \phi_k\|_{L^{\infty}(\Omega)}\leq c(\Omega,n)\lambda_k^{n/4},
\end{equation}
where the constant $c(\Omega,n)>0$.

Since $p^{\Omega}_\gamma\in L^2(\Omega\times \Omega,\gamma)$, for every $t>0$, then, for a.e. $x\in \Omega$, and for every fixed $t>0$, we have $p^{\Omega}_\gamma (x,z;t)\in L^2(\Omega,\gamma)$. Since $(\phi_k)_{k\in \N}$ is an orthonormal basis for $L^2(\Omega,\gamma)$, we have
\[
 p^{\Omega}_\gamma (x,z;t)=\sum^{\infty}_{k=1} \left(\int_\Omega p^{\Omega}_\gamma (x,z;t)\phi_k(z)d\gamma(z) \right) \phi_k(z).
 \]
On the other hand, by the kernel representation of the semigroup $e^{-tL_\Omega}$,
$$
\int_\Omega p^{\Omega}_\gamma (x,z;t)\phi_k(z)d\gamma(z)=[e^{-tL_\Omega}\phi_k ](x).
$$
Hence, by \eqref{Section5-1} we have
\begin{equation}\label{Add-5-1}
 p^{\Omega}_\gamma (x,z;t)=\sum^{\infty}_{k=1}[e^{-tL_\Omega}\phi_k ](x)\phi_k(z)=\sum^{\infty}_{k=1} e^{-t\lambda_k}\phi_k(x) \phi_k(z)<+\infty.
\end{equation}
Letting $z=x$ in \eqref{Add-5-1}, we have
\[
 p^{\Omega}_\gamma (x,x;t)=\sum^{\infty}_{k=1}[e^{-tL_\Omega}\phi_k ](x)\phi_k(z)=\sum^{\infty}_{k=1} e^{-t\lambda_k}\phi_k(x)^2<+\infty.
\]
Since for every $k$, $\| \phi_k\|_{L^2(\Omega,\gamma)}=1$, for any positive integer $N$ we can deduce that
\begin{equation}\label{Add-5-2}
\sum^{N}_{k=1} e^{-t\lambda_k}\leq \int_{\Omega} p^\Omega_\gamma(x,x;t)d\gamma(x)\leq  \int_{\R^n}  p_\gamma (x,x;t)d\gamma(x)<+\infty.
\end{equation}
Therefore, by the monotone convergence theorem, we have
\[
\mathop{\sum}^{\infty}_{k=1}e^{-t\lambda_k}< \infty,
\]
for all $t>0$. This completes the proof.
\end{proof}

\begin{remark}\label{r1} The classical bootstrap argument of elliptic partial differential equations could be used to get the $L^\infty$-estimate for the eigenfunctions. Furthermore, the Schauder interior estimates permit us to get the estimate of higher derivatives, in particular for the first and second derivatives (cf. \cite{Book-Gilbarg-Trudinger}).
\end{remark}

In the following, we show the connection between the basis $\{\phi_k\}_{k\geq 1}$ of eigenfunctions and the kernel of Ornstein-Uhlenbeck operator, that is, the spectral representation of the kernel, which implies that $p^\Omega_\gamma$ is also a solution of parabolic equation \eqref{parabolic-equation1}.

\begin{proposition}\label{p1}
Let $\Omega$ be an open, bounded set in $\R^n$, with Lipschitz boundary. Then
\begin{equation}\label{aggiunta 1000}
p^\Omega_{\gamma}(x,z;t)=\mathop{\sum}_{k=1}^{\infty} e^{-\lambda_k t}\phi_k(x)\phi_k(z),\quad\forall\,(x,z)\in\Omega\times\Omega,\, \forall\, t>0.
\end{equation}
Moreover, $p^\Omega_\gamma$ is the unique solution to problem \eqref{parabolic-equation1}.
\end{proposition}

\begin{proof}[\bf Proof]
By the proof of Proposition \ref{p2}, we know that
\[
p_\gamma^\Omega(x,z;t)=\mathop{\sum}_{k=1}^{\infty} e^{-\lambda_k t}\phi_k(x)\phi_k(z)
\]
for a.e. $x,y\in\Omega$ and for every $t>0$. On the other hand, again by Proposition \ref{p2}, the series
$$
\mathop{\sum}_{k=1}^{\infty} e^{-\lambda_k t}\phi_k(x)\phi_k(z)
$$
is uniformly convergent on $\Omega \times \Omega \times [a,+\infty)$ for any $a>0$. Hence \eqref{aggiunta 1000} holds. Using Proposition \ref{p2} and \cite[Theorem 6.2]{Book-Gilbarg-Trudinger}, it can be proved that the series generated by the first and second partial derivatives are still uniformly convergent. Hence
\begin{equation}\label{4-4}
-\Delta_x p_\gamma^\Omega+(x,\nabla p^\Omega_\gamma)=\partial_t p^\Omega_\gamma,\quad \mathrm{for\ all}\ x,z\in \Omega,\ t>0.
\end{equation}

Now let $\psi \in L^2(\Omega,\gamma)$; then
\[
\psi(z)=\mathop{\sum}^{\infty}_{j=1} (\psi, \phi_j)\phi_j(z),
\]
where, for brevity, we set
$$
(\psi, \phi_j)=\int_\Omega \psi\phi_j d\gamma.
$$
Since the sequence $\phi_k$ are mutually orthogonal, we have
 \begin{align*}
\int_\Omega p_\gamma^\Omega(x,z;0) \psi(z)d\gamma(z)&=\int_\Omega \sum_{k=1}^{\infty} \sum_{j=1}^{\infty}  (\psi, \phi_j) \phi_k(x) \phi_k(z)\phi_j(z)d\gamma(z)\\
&=\sum_{k=1}^{\infty}  (\psi, \phi_k)\phi_k(x)=\psi(x),
\end{align*}
which proves that $p_\gamma^\Omega(x,z,0)d\gamma(z)=\delta(x-z)$.
Finally, since $\phi_k\in W^{1,2}_0(\Omega,\gamma)$, then for every $t>0$, $p^\Omega_\gamma(x,z;t)=0$ if $x\in \partial \Omega$ or $z\in \partial \Omega$.

We have then proved that $p^\Omega_\gamma$ is a classical solution of problem \eqref{parabolic-equation1}. We also note that such a solution is unique by the maximum principle for parabolic equations.
\end{proof}

\begin{proof}[\bf Proof of Theorem \ref{T1}.]
by Proposition \ref{p1}, we can deduce that $u(x,t)$, which is defined in \eqref{1}, is a solution of \eqref{H-PDE2}. The conclusion follows from the second claim of Theorem \ref{Lemma-1}.
\end{proof}

\section{Application to the first eigenfunction of Ornstein-Uhlenbeck operator}\label{Se5}

In this section, we give new proofs to Theorems \ref{Application1} and \ref{Application2}.

Let $\Omega$ be a bounded open subset of $\R^n$, with Lipschitz boundary, and $u$ be a solution of \eqref{Lambda-Gauss-intro}. This is equivalent to say that $\lambda_\gamma=\lambda_1$ is the first Dirichlet eigenvalue of the operator $-L$ in $\Omega$, and $u=\phi_1$ is the corresponding first eigenfunction.

\begin{proof}[\bf Proof of Theorem \ref{Application1}]
We will prove that
$$
\lim_{t\to+\infty}p_\gamma^{\Omega}(x,z;t)e^{\lambda_1 t}e^{-|z|^2/2}=u(x)u(z)e^{-|z|^2/2}
$$
for every $x,z\in\Omega$. By Theorem \ref{Lemma-1}, this implies that $u$ is log-concave. The previous relation is equivalent to prove that
$$
\lim_{t\to+\infty}p_\gamma^{\Omega}(x,z;t)e^{\lambda_1 t}=u(x)u(z)
$$
for every $x,z\in\Omega$. By \eqref{aggiunta 1000},
\[
p^{\Omega}_\gamma\cdot e^{\lambda_1 t}-u(x)u(z)=\mathop{\sum}_{k=2}^{\infty} e^{(-\lambda_k+\lambda_1) t} \phi_{k}(x)\phi_{k}(z), \quad\forall\,(x,z)\in\Omega\times\Omega,\, \forall\, t>0.
\]
By the Cauchy-Schwarz inequality,
\begin{align*}
|p^\Omega_{\gamma}(x,z;t)&\cdot e^{\lambda_1t}-u(x)u(z)|\\
&\leq  \left( \mathop{\sum}_{k=2}^{\infty}  e^{(-\lambda_k +\lambda_1) t}\phi^2_{k}(x) \right)^{1/2} \left( \mathop{\sum}_{k=2}^{\infty}  e^{(-\lambda_k +\lambda_1) t}\phi^2_{k}(z) \right)^{1/2} .
\end{align*}
In order to get bounds on each of these square roots, we use again the spectral representation with $z=x$ and $t=1$, and Proposition \ref{p2}:
\begin{equation}\label{Kac-(9.9)}
\mathop{\sum}_{k=1}^{\infty}  e^{-\lambda_k }\phi^2_{k}(x)=p^\Omega_{\gamma}(x,x;1)\leq c_0 (1-e^{-2})^{-n/2},
\end{equation}
where the positive constant $c_0$ depends on $\Omega$ and $n$. For $t>0$, we define $M(t):=\sup_{k\geq 2}e^{-\lambda_kt+\lambda_1t+\lambda_k}$, we have
\begin{align*}
|p_{\gamma}^{\Omega}(x,z;t)&\cdot e^{\lambda_1t}-u(x)u(z)|\\
&\leq  M(t)\left( \mathop{\sum}_{k=2}^{\infty}  e^{-\lambda_k }\phi^2_{k}(x) \right)^{1/2} \left( \mathop{\sum}_{k=2}^{\infty} e^{-\lambda_k  }\phi^2_{k}(z) \right)^{1/2} \\
&\leq M(t)c_0 (1-e^{-2})^{-n/2}.
\end{align*}
On the other hand, as the sequence $\lambda_k$ is increasing, for $t\ge1$, we have
$$
M(t)\le e^{-\lambda_2t+\lambda_1t+\lambda_2}.
$$
Hence
$$
\lim_{t\rightarrow \infty} M(t)=0.
$$
The proof is completed.
\end{proof}

\begin{proof}[\bf Proof of Theorem \ref{Application2}]
Let $\Omega$ be an open and bounded subset of $\R^n$; we define the trace function $Z_\Omega$ of $\Omega$ as follows
\[
Z_\Omega(t):=\int_\Omega p_\gamma^{\Omega}(x,x;t)d\gamma(x), \quad t>0.
\]
By Propositions \ref{p2} and \ref{p1}, we know that $Z_\Omega(t)<\infty$ for all $t>0$. Firstly, we observe that
\[
\int_{\Omega} p^\Omega_{\gamma}(x,x;t)d\gamma(x)=\mathop{\lim}_{m\rightarrow \infty}\int_\Omega\cdots \int_\Omega \mathop{\prod}_{k=1}^{m} p_\gamma \left(x_{k-1},x_k;\frac{t}{m} \right)d\gamma(x_1)\cdots d\gamma(x_{m-1}) d\gamma(x_{m}),
\]
where $x_0=x$, $x_m=x$. Define
\[
G_{m,\Omega}(t)=\int_\Omega\cdots \int_\Omega \mathop{\prod}_{k=1}^{m} p_\gamma \left(x_{k-1},x_k;\frac{t}{m} \right)d\gamma(x_1)\cdots d\gamma(x_{m-1}) d\gamma(x_{m}).
\]
Moreover, Theorem \ref{Lemma-1} tells us that
$$
(x,z)\,\rightarrow\, p_\gamma^\Omega(x,z;t)e^{-|z^2|/2},\quad x,z\in\Omega,
$$
is log-concave. So, using \cite[Corollary 3.4]{BL76}, we get that for every positive integer $m$, we have
\[
G_{m,\Omega_s}(t)\geq G_{m,\Omega_0}(t)^{1-s} G_{m,\Omega_1}(t)^{s},\quad\forall\, s\in [0,1].
\]
Then, letting $m\rightarrow \infty$, we have
\begin{equation}\label{aggiunta 10000}
Z_{\Omega_s}(t)\geq Z_{\Omega_0}(t)^{1-s} Z_{\Omega_1}(t)^{s},\quad\forall\, s\in [0,1].
\end{equation}
Furthermore, as shown in proof of Theorem \ref{Application1}, we have
\[
p_\gamma^{\Omega}(x,x;t)e^{\lambda_1 t}\rightarrow u(x)^2 ,\quad \mathrm{as}\ t\rightarrow +\infty.
\]
where $x\in \overline{\Omega}$. Since $\| u\|_{L^2(\Omega,\gamma)}=1$, we have
\[
 e^{\lambda_1 t}\int_\Omega p_\gamma^{\Omega}(x,x;t) d\gamma(x) \rightarrow 1,\quad \mathrm{as}\ t\rightarrow +\infty.
\]
Therefore, we deduce that
\[
\lambda_1(\Omega_0)=-\mathop{\lim}_{t\rightarrow \infty} t^{-1} \log Z_{\Omega_0}(t),
\]
and similar relations hold for $\Omega_1$ and $\Omega_s$. By \eqref{aggiunta 10000}, the proof is complete.
\end{proof}

\section{Appendix}

In appendix, we give more details about semigroup theory of Ornstein-Uhlebeck in the class of bounded Lipschitz domains. When we prove the spectral representation of the integral kernel, we assume a priori that the semigroup $e^{-tL_\Omega}$ generated by $L_\Omega$ has a kernel representation. As previously mentioned in this paper, the reference \cite[pages 424-428]{Wood} has already proved that the Ornstein-Uhlenbeck operator generates an analytic $\mathcal{C}_0$-semigroup in the class of bounded Lipschitz domain. In the class of convex bodies, those properties of the semigroup were proved by Cappa \cite{Cappa2019}, and more details were presented in his PhD thesis \cite{Cappa2019-2}. We note that Cappa's proof does not require convexity. The main proof follows the book by Davies \cite{DaviesBook} and the handbook by W. Arendt \cite[Chapter 7]{Arendt-handbook}, so, we can deduce that the semigroup has a representation as an integral kernel (see for instance \cite[Page 422, Theorem 2.18]{Wood}). Precisely we can retrieve the properties listed below.

\begin{proposition}
Let $\Omega\subset\mathbb{R}^n$ be a bounded Lipschitz domain and let $-L_\Omega$ denote the Dirichlet Ornstein-Uhlenbeck operator, with associated $\mathcal{C}_0$-semigroup $T(t)=e^{-tL_\Omega}$ in $L^2(\Omega,\gamma)$, $t>0$. Then:
\begin{enumerate}
    \item For every $t>0$, $T(t):L^2(\Omega,\gamma)\to L^\infty(\Omega)$ is bounded and
    \[
    \|T(t)\|_{L^2\to L^\infty} \le C\,{(1-e^{-2t})}^{-n/4},
    \]
    where $C>0$ depends only on $\Omega$ and $n$.
    \item For each $t>0$ there exists a measurable kernel
    \[
    p^{\Omega}_\gamma:\Omega\times\Omega\to [0,\infty)
    \]
    such that for all $f\in L^2(\Omega,\gamma)$ and a.e.\ $x\in\Omega$,
    \[
    (T(t)f)(x) = \int_\Omega p^{\Omega}_\gamma(x,y;t)\, f(y)\,d\gamma(y),
    \]
    with
    \[
    \operatorname*{ess\,sup}_{x\in\Omega} \|p^{\Omega}_\gamma(x,\cdot,t)\|_{L^2(\Omega,\gamma)} \le \|T(t)\|_{L^2\to L^\infty} \le C\, {(1-e^{-2t})}^{-n/4}.
    \]
    \item The kernel has the following properties.
    \begin{itemize}
        \item Symmetry: $p^\Omega_\gamma(x,y;t)=p^\Omega_\gamma(y,x)$ for a.e.\ $x,y\in\Omega$.
        \item Positivity: $p^\Omega_\gamma(x,y;t)\ge 0$ for a.e.\ $x,y\in\Omega$.
        \item Semigroup property, i.e.
        \[
        p^\Omega_\gamma(x,y;t+s) = \int_\Omega p^\Omega_\gamma(x,z;t)\,p^\Omega_\gamma(z,y;t)\,d\gamma(z), \quad \text{for all } s,t>0,
        \]
        for a.e.\ $x,y\in\Omega$.
    \end{itemize}
\end{enumerate}
\end{proposition}

Here, the kernel only belongs to $L^{\infty}(\Omega\times \Omega,\gamma)$. Taking into account the spectral estimates in Section \ref{Se4}, we can get its smoothness, i.e., $p^\Omega_\gamma\in C^{\infty}(\Omega \times \Omega \times (0,+\infty)) \cap C^{0}(\overline{\Omega}\times \overline{\Omega}\times (0,\infty))$.

\bibliographystyle{amsplain}

\end{document}